\documentclass[12 pt]{amsart}
\newtheorem*{theorem}{Theorem}
\newtheorem*{lemma}{Lemma}

\begin{document}
\title{On Doubly Periodic Phases}
\author{Marius B. Stefan}
\address{Waltham, MA 02453, USA}
\email{stefanmb@member.ams.org}
\begin{abstract}
Every meromorphic function on $\mathbb{C}$ with doubly periodic
phase is equal to an elliptic function multiplied by a meromorphic
function determined by the periods.
\end{abstract}
\maketitle

The purpose of this note is to completely characterize the
meromorphic functions on $\mathbb{C}$ that have doubly periodic
phases. In particular, our result shows that the characterization
suggested by G. Semmler and E. Wegert in \cite{SeWe} is too narrow.
The phase of a function $f(z)$ defined on an open set
$D\subset\mathbb{C}$ with values in $\widehat{\mathbb{C}}$ is the
function $\frac{f(z)}{|f(z)|}$ defined on $\{z\in
D:f(z)\in\widehat{\mathbb{C}}\setminus\{0,\infty\}\}$ with values in
$\mathbb{T}$.

We recall the definition of the Weierstrass sigma-function
(\cite{We2}). If $p_1,p_2\in\mathbb{C}$ are linearly independent
over $\mathbb{R}$, we denote by $L$ the $\mathbb{Z}$-module
generated by $p_1,p_2$ and we put $\omega=\frac{p_2}{p_1}$. Since
the series
\begin{equation*}
\sum\limits_{\lambda\in L\setminus\{0\}}\frac{1}{|\lambda|^3}
\end{equation*}
is convergent, the Weierstrass product (\cite{We1})
\begin{equation*}
z\prod_{\lambda\in L\setminus\{0\}}e^{\frac{z}{\lambda}+\frac{z^2}{2
\lambda^2}}\left(1-\frac{z}{\lambda}\right)
\end{equation*}
defines an entire function $\sigma(z)$ whose set of zeros is
precisely $L$, and which is known as the Weierstrass sigma-function.
The following transformation property is probably well-known but we
include a proof, for completeness.
\begin{lemma}
For every $\xi_0\in\mathbb{C}$ and $j\in\{1,2\}$ the Weierstrass
sigma-function satisfies
\begin{equation*}
\frac{\sigma(z)}{\sigma(-\xi_0+z)}\cdot
\frac{\sigma(-\xi_0+z+p_j)}{\sigma(z+p_j)}=e^{v_j},
\end{equation*}
where
\begin{equation*}
v_j=-\frac{3\xi_0}{p_j}+\xi_0p_j^2 \sum_{\lambda\in
L\setminus\{0,-p_j\}}\frac{1}{\lambda(\lambda+p_j)^2}.
\end{equation*}
\end{lemma}
\begin{proof}
For every $\xi_0\in\mathbb{C}$ and $j\in\{1,2\}$ we have
\begin{align*}
\frac{\sigma(z)}{\sigma(-\xi_0+z)}&=\frac{z}{-\xi_0+z}\prod_{\lambda\in
L\setminus\{0\}}e^{\frac{\xi_0}{\lambda}-\frac{\xi_0^2}{2\lambda^2}
+\frac{\xi_0z}{\lambda^2}}\frac{\lambda-z}{\lambda+\xi_0-z}\\&=
\frac{z(z+p_j)e^{-\frac{\xi_0}{p_j}
-\frac{\xi_0^2}{2p_j^2}+\frac{\xi_0z}{p_j^2}}}{(-\xi_0+z)(-\xi_0+z+p_j)}\\
&\cdot\prod_{\lambda\in
L\setminus\{0,-p_j\}}e^{\frac{\xi_0}{\lambda}-\frac{\xi_0^2}{2\lambda^2}
+\frac{\xi_0z}{\lambda^2}}\frac{\lambda-z}{\lambda+\xi_0-z}
\end{align*}
and
\begin{align*}
\frac{\sigma(z+p_j)}{\sigma(-\xi_0+z+p_j)}&=\frac{z+p_j}{-\xi_0+z+p_j}
\prod_{\lambda\in
L\setminus\{0\}}e^{\frac{\xi_0}{\lambda}-\frac{\xi_0^2}{2\lambda^2}
+\frac{\xi_0(z+p_j)}{\lambda^2}}\frac{\lambda-p_j-z}{\lambda-p_j+\xi_0-z}\\
&=\frac{z+p_j}{-\xi_0+z+p_j}\\&\cdot\prod_{\lambda\in
L\setminus\{-p_j\}}e^{\frac{\xi_0}{\lambda+p_j}-\frac{\xi_0^2}{2(\lambda+p_j)^2}
+\frac{\xi_0(z+p_j)}{(\lambda+p_j)^2}}\frac{\lambda-z}{\lambda+\xi_0-z}
\\&= \frac{z(z+p_j)e^{\frac{2\xi_0}{p_j}
-\frac{\xi_0^2}{2p_j^2}+\frac{\xi_0z}{p_j^2}}}{(-\xi_0+z)(-\xi_0+z+p_j)}\\
&\cdot\prod_{\lambda\in
L\setminus\{0,-p_j\}}e^{\frac{\xi_0}{\lambda+p_j}-\frac{\xi_0^2}{2(\lambda+p_j)^2}
+\frac{\xi_0(z+p_j)}{(\lambda+p_j)^2}}\frac{\lambda-z}{\lambda+\xi_0-z}
\end{align*}
therefore
\begin{align*}
\frac{\sigma(z)}{\sigma(-\xi_0+z)}&\cdot
\frac{\sigma(-\xi_0+z+p_j)}{\sigma(z+p_j)}
\\&=e^{-\frac{3\xi_0}{p_j}+
\sum\limits_{\lambda\in
L\setminus\{0,-p_j\}}\left(\frac{\xi_0}{\lambda}-\frac{\xi_0}{
\lambda+p_j}-\frac{\xi_0^2}{2\lambda^2}+\frac{\xi_0^2}{2(\lambda+p_j)^2}
+\frac{\xi_0z}{\lambda^2}-\frac{\xi_0(z+p_j)}{(\lambda+p_j)^2}\right)}\\&
=e^{u_jz+v_j-\frac{1}{2}\xi_0u_j},
\end{align*}
where
\begin{equation*}
u_j=\xi_0\sum_{\lambda\in
L\setminus\{0,-p_j\}}\left(\frac{1}{\lambda^2}-\frac{1}{(\lambda+p_j)^2}
\right).
\end{equation*}
For $\omega\in\mathbb{C}\setminus\mathbb{R}$, we define
$f_1(\omega)$ by
\begin{align*}
p_1^2u_1&=p_1^2\xi_0\sum_{\lambda\in
L\setminus\{0,-p_1\}}\left(\frac{1}{\lambda^2}-\frac{1}{(\lambda+p_1)^2}
\right)\\&=\xi_0\sum_{(m,n)\in\mathbb{Z}^2\setminus\{(0,0),(-1,0)\}}\left(\frac{1}
{(m+n\omega)^2}-\frac{1}{(m+1+n\omega)^2}\right)\\
&=:f_1(\omega)
\end{align*}
and notice that
\begin{align*}
p_2^2u_2&=p_2^2\xi_0\sum_{\lambda\in
L\setminus\{0,-p_2\}}\left(\frac{1}{\lambda^2}-\frac{1}{(\lambda+p_2)^2}
\right)\\&=\xi_0\sum_{(m,n)\in\mathbb{Z}^2\setminus\{(0,0),(0,-1)\}}
\left(\frac{1}{(\frac{m}{\omega}+n)^2}-\frac{1}{(\frac{m}{\omega}+n+1)^2}\right)\\
&=f_1\left(\frac{1}{\omega}\right).
\end{align*}
It now suffices to show that $f_1(\omega)=0$ for every
$\omega\in\mathbb{C}\setminus\mathbb{R}$:
\begin{align*}
f_1(\omega)&=\xi_0\sum_{n\in\mathbb{Z}\setminus\{0\}}\sum_{m\in\mathbb{Z}}\left(
\frac{1}{(m+n\omega)^2}-\frac{1}{(m+1+n\omega)^2}\right)\\
&+\xi_0\sum_{m\in\mathbb{Z}\setminus\{-1,0\}}\left(\frac{1}{m^2}-\frac{1}{(m+1)^2}
\right)\\&=0.
\end{align*}
\end{proof}
\begin{theorem}
Let $f(z)$ be a nonconstant function meromorphic on $\mathbb{C}$. If
the phase of $f(z)$ is doubly periodic with primitive periods
$p_1,p_2$ and
$\omega=\frac{p_2}{p_1}\in\mathbb{C}\setminus\mathbb{R}$, then there
exists an elliptic function $g(z)$ with periods $p_1,p_2$ such that
\begin{equation*}
f(z)=e^{az}g(z)\frac{\sigma(z)}{\sigma(-\xi_0+z)},
\end{equation*}
where $\xi_0\in\mathbb{C}$, and $a\in\mathbb{C}$ satisfies
$\Im(ap_j)=\Im(v_j)+2m_j\pi$, $j\in\{1,2\}$, for some
$m_1,m_2\in\mathbb{Z}$.
\end{theorem}
\begin{proof} We shall use the observation (\cite{SeWe}) that if the
phase of a meromorphic function $f(z)$ on $\mathbb{C}$ has a period
$p$, then $f(z+p)=e^\alpha f(z)$ for some constant
$\alpha\in\mathbb{R}$. Let $\alpha_1,\alpha_2$ be two real numbers
such that $f(z+p_1)=e^{\alpha_1}f(z)$ and
$f(z+p_2)=e^{\alpha_2}f(z)$. Let $F$ denote the parallelogram with
the two vectors $p_1,p_2$ as adjacent sides. The function
$\frac{f'(z)}{f(z)}$ is obviously elliptic with periods $p_1,p_2$
and its integral over $F$ is therefore $0$. In particular, $f(z)$
has the same number of zeros and poles inside $F$. Let $\Xi$ be the
multiset (we take into consideration the multiplicities that occur)
consisting of all zeros of $f(z)$ inside $F$ and $\Gamma$ be the
multiset consisting of all poles of $f(z)$ inside $F$. We then have
\begin{align*}
\sum_{\xi\in\Xi}\xi-\sum_{\gamma\in\Gamma}\gamma&=\frac{1}{2\pi
i}\int_F\frac{zf'(z)}{f(z)}dz\\ &=\frac{1}{2\pi
i}\int_0^{p_1}(z-(z+p_2))\frac{f'(z)}{f(z)}dz\\
&-\frac{1}{2\pi
i}\int_0^{p_2}(z-(z+p_1))\frac{f'(z)}{f(z)}dz\\
&=-\frac{p_2}{2\pi i}\int_0^{p_1}d\log f(z) +\frac{p_1}{2\pi
i}\int_0^{p_2}d\log f(z)\\ &=-\frac{p_2}{2\pi i}(\alpha_1+2n_1\pi
i)+\frac{p_1}{2\pi i}(\alpha_2+2n_2\pi i)\\
&\equiv-\xi_0\,\,\,\mbox{mod}\,\,\,L
\end{align*}
for some $n_1,n_2\in\mathbb{Z}$, where $\xi_0$ is the unique number
inside $F$ that is congruent to $\frac{\alpha_1p_2-\alpha_2p_1}{2\pi
i}$ modulo $L$. The congruence above is equivalent to
\begin{equation*}
\sum_{\xi\in\{\xi_0\}\cup\Xi}\xi-\sum_{\gamma\in\{0\}\cup\Gamma}
\gamma\equiv 0\,\,\,\mbox{mod}\,\,\,L
\end{equation*}
hence, by Abel's theorem (\cite{Ab1, Ab2, Ab3}), there exists an
elliptic function $g(z)$ with periods $p_1,p_2$ such that its
multisets of zeros and poles inside $F$ are $\{\xi_0\}\cup\Xi$ and
$\{0\}\cup\Gamma$, respectively. The sets of zeros and poles of the
meromorphic function $\frac{g(z)}{f(z)}$ are then $\xi_0+L$ and $L$,
respectively. Since the meromorphic function
$\frac{\sigma(-\xi_0+z)}{\sigma(z)}$ has exactly the same zeros and
poles as $\frac{g(z)}{f(z)}$, there exists an entire function $h(z)$
such that
\begin{equation*}
f(z)=e^{h(z)}g(z)\frac{\sigma(z)}{\sigma(-\xi_0+z)}.
\end{equation*}
The conditions $f(z+p_j)=e^{\alpha_j}f(z)$, $j\in\{1,2\}$, now imply
\begin{equation*}
e^{h(z+p_j)}\frac{\sigma(z+p_j)}{\sigma(-\xi_0+z+p_j)}=e^{\alpha_j+h(z)}
\frac{\sigma(z)}{\sigma(-\xi_0+z)}
\end{equation*}
hence, using the Lemma,
\begin{align*}
e^{h(z+p_j)-h(z)-\alpha_j}&=\frac{\sigma(z)}{\sigma(-\xi_0+z)}\cdot
\frac{\sigma(-\xi_0+z+p_j)}{\sigma(z+p_j)}\\&=e^{v_j}.
\end{align*}
The function $h'(z)$ is therefore doubly periodic and has no poles,
hence $h'(z)$ is constant. So $h(z)=az+b$ for some
$a,b\in\mathbb{C}$. Moreover, $ap_j-\alpha_j=v_j+2m_j\pi i$ for
$j\in\{1,2\}$ and some $m_1,m_2\in\mathbb{Z}$. We must have
$ap_j-v_j-2m_j\pi i=\alpha_j\in\mathbb{R}$, $j\in\{1,2\}$, that is,
\begin{equation*}
\Im(ap_j)=\Im(v_j)+2m_j\pi,\,\,\,\,j\in\{1,2\}.
\end{equation*}
\end{proof}

\end{document}